\documentclass[11pt,a4paper]{amsart}
\usepackage{amsmath,amsthm,amssymb,amsfonts}
\usepackage{graphicx}
\usepackage{epstopdf}
\usepackage{color}
\usepackage{hyperref}
\usepackage{paralist}
\usepackage{caption}

\usepackage{thm-restate}
\usepackage{enumitem}
\newcommand\undermat[2]{%
  \makebox[0pt][l]{$\smash{\underbrace{\phantom{%
    \begin{matrix}#2\end{matrix}}}_{\text{$#1$}}}$}#2}

\newtheorem{theorem}{Theorem}
\newtheorem{proposition}[theorem]{Proposition} 
\newtheorem{lemma}[theorem]{Lemma} 
\newtheorem{corollary}[theorem]{Corollary}

\theoremstyle{remark}

\theoremstyle{definition}

\newcommand{\ie}{i.e.,~} %

\def\cQ{\mathcal{Q}}
\def\cR{\mathcal{R}}

\def\KK{\mathbb{K}}
\def\NN{\mathbb{N}}
\def\QQ{\mathbb{Q}}
\def\RR{\mathbb{R}}

\def\RR{\mathbb{R}}

\def\aa{\alpha}
\def\bb{\beta}

\newcommand{\vecone}{\boldsymbol{1}}%

\DeclareMathOperator{\conv}{conv}

\DeclareMathOperator{\verts}{vert}

\DeclareMathOperator{\pyr}{pyr}
\newcommand{\join}{\star}

\DeclareMathOperator{\xc}{xc}

\newcommand{\rs}{\cR}

\newcommand{\defn}[1]{\emph{\color{blue} #1}} %
\newcommand{\set}[2]{\ensuremath{\left\{#1\,\middle|\,#2\right\}}} 
\newcommand{\ffloor}[2]{\left\lfloor{\frac{#1}{#2}}\right\rfloor} %
\newcommand{\fceil}[2]{\left\lceil {\frac{#1}{#2}} \right\rceil} %

\newcommand{\inv}[1]{{{#1}^{-1}}}
\newcommand{\car}[2]{\Delta_{#1}\times\Delta_{#2}}
\newcommand{\dsum}[2]{\Delta_{#1}\oplus\Delta_{#2}}
\newcommand{\pcar}[3]{\pyr_{#1}(\Delta_{#2}\times\Delta_{#3})}
\newcommand{\psum}[3]{\pyr_{#1}(\Delta_{#2}\oplus\Delta_{#3})}

\newcommand{\pc}{\pcar{k}{n}{m}}
\newcommand{\ps}{\psum{k}{n}{m}}

\newcommand{\D}[2]{D(#1,#2)}
\newcommand{\Dab}{\D{\alpha}{\beta}}
\newcommand{\F}[2]{\Phi(#1,#2)}

\newcommand{\pd}{\pi_d}
\newcommand{\pdp}[1]{\pi_{#1}}

\title{Extension complexity of polytopes with few vertices or facets}

\author{Arnau Padrol}
\address{
Sorbonne Universit\'es, Universit\'e Pierre et Marie Curie (Paris 6), Institut de Math\'ematiques de Jussieu - Paris Rive Gauche (UMR 7586), Paris, France
}
\email{arnau.padrol@imj-prg.fr}
\thanks{This research was supported by the DFG
Collaborative Research Center SFB/TR~109 ``Discretization
in Geometry and Dynamics'' as well as the program PEPS
Jeunes Chercheur-e-s 2016 of the INSMI (CNRS)}

\begin{document}

\begin{abstract}
We study the extension complexity of polytopes with few vertices or facets. On the one hand, 
we provide a complete classification of $d$-polytopes with at most $d+4$ vertices according to their extension complexity: Out of the super-exponentially
many $d$-polytopes with $d+4$ vertices, all have extension complexity $d+4$ except for some families of size $\theta(d^2)$. On the other
hand, we show that generic realizations of simplicial/simple $d$-polytopes with $d+1+\alpha$ vertices/facets have extension complexity at least $2 \sqrt{d(d+\alpha)} -d + 1$,
which shows that for all $d>(\frac{\alpha-1}{2})^2$ there are $d$-polytopes with $d+1+\alpha$ vertices or facets and extension complexity $d+1+\alpha$.
\end{abstract}

\maketitle

\section{Introduction}

The \defn{extension complexity} of a polytope $P$, denoted by $\xc(P)$, is the minimal number of facets of a
polytope that can be linearly projected onto~$P$. 
This concept was introduced in combinatorial optimization, 
since polytopes with small extension complexity correspond to  
optimization problems that have efficent formulations as linear programs (see~\cite{ConfortiCornuejolsZambelli2010,Kaibel2011} for recent surveys); but 
it is also very relevant in many other areas thanks to its reformulation in terms of nonnegative matrix factorizations~\cite{Yannakakis1991}.
Despite being the subject of extensive research, it is still far from being well understood, and many basic questions are still unresolved (see the reports \cite{Dagstuhl2013,Dagstuhl2015} for the latest results and many open problems). 
In particular, there are very few polytopes for which the exact extension complexity is known. Examples are 
cubes, Birkhoff polytopes and bipartite matching polytopes~\cite{FioriniKaibelPashkovichTheis2013}, and hypersimplices~\cite{GrandePadrolSanyal2016}.

In this paper, we study the extension complexity of high-dimensional polytopes 
with few vertices; namely, $d$-polytopes with $d+1+\alpha$ 
vertices for fixed (small) values of $\alpha\in\NN$. Of course, since the extension 
complexity is preserved by polar duality (cf. \cite[Proposition~2.8]{GouveiaParriloThomas2013}), this is equivalent to studying the extension complexity of polytopes with few facets.

Our main motivation for this project
was to provide examples of high-dimensional polytopes for which the explicit determination of the
extension complexity is still treatable, in order to obtain a ground set for testing open problems and looking for examples and counterexamples.
The first of the goals is fulfilled: We provide a complete classification of $d$-polytopes with at most $d+4$ vertices according to their extension complexity in Theorem~\ref{thm:main}. However, in view of the final classification, 
it is not clear that this family will be a rich source for useful objects.

Complete understanding of the extension complexity of the next natural family, 
$d$-polytopes with $d+5$ vertices, seems out of reach right now, specially if we take into account 
that this was a non-trivial problem already for $d=2$. This was studied by Shitov,
who proved that every polygon with $7$ vertices has extension complexity~$6$, which induces an upper bound of $\fceil{6n}{7}$ for the extension complexity of an $n$-gon \cite{Shitov2014} (see also \cite{PadrolPfeifle2015} and \cite{Shitov2014a}).

If there was an $\alpha$ such that every $d$-polytope with $d+1+\alpha$ vertices has extension complexity at most $d+\alpha$, this would similarly provide upper bounds for the extension complexity of $d$-polytopes in terms of their number of vertices. (It is still an open problem whether for each $n$ there exist $d$-polytopes with $n$ vertices and extension complexity $n$; for which we only know the answer, in the negative, for the case $d\leq 2$.) However, as we will see, this cannot happen (with an $\alpha$ independent of $d$). Seeing the lower bounds for the extension complexity of generic $d$-polytopes with $d+1+\alpha$ vertices from Section~\ref{sec:bounds}, it might be as well that really interesting behaviors only occur with polytopes with many vertices.

In what remains of this introduction, we present our main results: the classification of $d$-polytopes with up to $d+4$ vertices and the lower bounds for the extension complexity of generic polytopes. The first is developed in Section~\ref{sec:d+4vertices}, and the latter in Section~\ref{sec:bounds}. The results are independent and use completely different tools, and hence can be read separately.

\subsection{The ``threshold for counterexamples''}

Many properties shared by $d$-polytopes with at most $d+3$ vertices start failing for $d$-polytopes with $d+4$ vertices (in a similar way as properties of polytopes of dimension at most~$3$ usually start failing for $4$-polytopes).
This led Sturmfels to call $d$-polytopes with $d+4$ vertices the ``threshold for counterexamples''~\cite{Sturmfels1988}.

On the one hand, the number of combinatorial types of $d$-polytopes with $d+4$ vertices grows super-exponentially in $d$: There are at least ${d}^{\frac{3}{2}d(1-o(1))}$ (labeled) neighborly $d$-polytopes with $d+4$ vertices~\cite{Padrol2013}, and hence at least ${d}^{\frac{1}{2}d(1-o(1))}$ if we count unlabeled combinatorial classes. In contrast, there is only one $d$-polytope with $d+1$ vertices, the simplex; there are $\lfloor\frac{d^2}{2}\rfloor$ $d$-polytopes with $d+2$ vertices~\cite[Theorem~6.1.4]{Gruenbaum}; and there are $O(\frac{\gamma^d}{d})$ $d$-polytopes with $d+3$-vertices, where $\gamma\sim2.83$~\cite{Fusy2006}.

Moreover, the space of realizations of $d$-polytopes with $d+4$ vertices can also be very complicated.
Indeed, $d$-polytopes with $d+4$ vertices already exhibit Mn\"ev's Universality Phenomenon (\cite{Mnev1988}, see also \cite[Sections 6.5 and 6.6]{Ziegler}). Roughly speaking, this means that their realization spaces can be as complicated as any (primary basic) semi-algebraic set. This implies in particular that these realization spaces can have the homotopy type of any simplicial complex and that there are $d$-polytopes with $d+4$ vertices that cannot be realized with rational coordinates (actually, for any finite algebraic extension $\KK$ of $\QQ$, there is a polytope that cannot be realized over~$\KK$). In contrast, the simplex has a unique realization up to affine transformation; all $d$-polytopes with $d+2$ vertices are projectively unique; and the realization space of every $d$-polytope with $d+3$ vertices is contractible and contains rational points.

One is tempted to ask whether $d$-polytopes with $d+4$ vertices are also the threshold of counterexamples for extension complexity.
This would have been of particular interest in the context of the long-term open question that asked whether the rational and real nonnegative rank always coincide~\cite{BeasleyLaffey2009,CohenRothblum1993}, and that has been recently disproved by two independent teams~\cite{CKMSW2016,Shitov2016}. Shitov has even proved a much stronger result, a universality theorem for nonnegative factorizations~\cite{Shitov2016b}.

As we will see, the answer is however no. We use the Projection Lemma~\ref{lem:projlem} (adapted from \cite{SanyalZiegler2010,Ziegler2004}, and strongly related to Gale transforms) 
 to analyze the extension complexity of $d$-polytopes with $d+4$ vertices. They all have complexity $d+4$ except for some sporadic instances that can be constructed via some elementary operations from a finite collection of polytopes. In particular, it is easy to compute the extension complexity of every $d$-polytope with $d+4$ vertices (or facets, by duality).

Here, and throughout the paper, we use $\Delta_d$ to denote a $d$-dimensional simplex; $P \oplus Q$ and $P \times Q$ to represent, respectively, the \defn{direct sum} and the \defn{Cartesian product} of the polytopes $P$ and $Q$; and $\pyr_k(P)$ to denote the \defn{$k$-fold pyramid} over $P$.
See \cite[Sec.~15.1.3]{HenkRichterGebertZiegler1997} for the corresponding definitions.

\begin{theorem}\label{thm:main}
Let $P$ be a $d$-polytope with $d+4$ vertices, then 
\begin{enumerate}[label=(\arabic*)]
\item\label{it:cased+2} $\xc(P)=d+2$ if and only if $P$ has $d+2$ facets.
\item $\xc(P)=d+3$ if and only if:
\begin{enumerate}[label=(\arabic{enumi}.\arabic*)]
\item\label{it:cased+3}
$P$ has $d+3$ facets, or
\item%
$P=\pi(Q)$, where $Q\cong \pcar{d-2}{1}{2}$ for some affine projection~$\pi$.  

In this case, either
\begin{enumerate}[label=(\arabic{enumi}.\arabic{enumii}.\arabic*)]
 \item\label{it:casedesarguian} $P=\pyr_k(Q)$ where $Q$ is a Desarguian hexagon (a hexagon with $\xc(Q)=5$), or
 \item\label{it:casepyramid} $P$ has a subset of $6$ vertices forming a triangular prism.
\end{enumerate}
 \end{enumerate}
\item $\xc(P)=d+4$ otherwise.
\end{enumerate}
\end{theorem}

The proof of this theorem is split across several results in Section~\ref{sec:d+4vertices}, which describe the combinatorics of the possible polytopes that can appear in each of the cases.

For a $d$-polytope $P$ with $d+4$ vertices:
\begin{enumerate}[label=(\arabic*)]
\item $\xc(P)=d+2$ if and only if $P$ is combinatorially equivalent to~$\pcar{d-4}{1}{3}$ (Lemma~\ref{lem:xc=d+2} and Proposition~\ref{prop:d+4verticesd+3facets});
\item $\xc(P)=d+3$ if and only if:
\begin{enumerate}[label=(\arabic{enumi}.\arabic*)]
\item $P$ is an iterated pyramid over one of 
the $8$ sporadic non-pyramidal $d$-polytopes with $d+4$ vertices and 
$d+3$ facets (Proposition~\ref{prop:d+4verticesd+3facets})
\item\label{it:caseprismproj} $P=\pi(Q)$, where $Q\cong \pcar{d-2}{1}{2}$ for some affine projection~$\pi$ (Proposition~\ref{prop:d+4classification}). 

In this case, either
\begin{enumerate}[label=(\arabic{enumi}.\arabic{enumii}.\arabic*)]
 \item%
 $P=\pyr_k(Q)$ where $Q$ is a Desarguian hexagon (Lemma~\ref{lem:desarguian}), or
 \item%
 $P$ has a subset of $6$ vertices forming a triangular prism (Lemma~\ref{lem:d+4classification}).
  Which means that either
 \begin{enumerate}[label=(\arabic{enumi}.\arabic{enumii}.\arabic{enumiii}.\arabic*)]
  \item\label{it:joinsumprod} $P\cong \pyr_k(\car{1}{2})\join(\dsum{n}{m}))$ (where $n+m+k=d-4$), or
  \item\label{it:pyrsuspLaw} $P$ can be obtained from $\car{1}{2}$ via the operations of \emph{pyramid}, \emph{one-point-suspension} or \emph{Lawrence extension} on an extra (projective) point.
\end{enumerate}
\end{enumerate}
\end{enumerate}
\item $\xc(P)=d+4$ otherwise.
\end{enumerate}

In the families corresponding to the cases \ref{it:cased+2}, \ref{it:cased+3} and \ref{it:joinsumprod}, the extension complexity is completely determined by the 
combinatorial type. Deciding the extension complexity in the cases \ref{it:casedesarguian} and \ref{it:pyrsuspLaw} amounts to checking whether certain lines are concurrent.

For fixed~$d$, families \ref{it:joinsumprod} and \ref{it:pyrsuspLaw} have a size quadratic in~$d$ (see Corollary~\ref{cor:number}). Hence, out of the super-exponentially many combinatorial types of $d$-polytopes with $d+4$ vertices, there are only $\theta(d^2)$ that have realizations with extension complexity smaller than $d+4$.

\subsection{Lower bounds on extension complexity}

In Section~\ref{sec:bounds} we provide a lower bound for the extension complexity of generic polytopes in terms of the dimension of their realization space. With it we conclude that for every $d,\alpha,\beta\geq 0$, 
\begin{itemize}
 \item a generic (simplicial) $d$-polytope $P$ with $d+1+\alpha$ vertices has extension complexity 
 \[\xc(P)\geq 2 \sqrt{d(d+\alpha)} -d + 1;\]
 \item a generic (simple) $d$-polytope $P$ with $d+1+\beta$ facets has extension complexity \[\xc(P)\geq 2 \sqrt{d(d+\beta)} -d + 1.\]
\end{itemize}
Here, \emph{generic} has to be understood as randomly drawn among all \emph{geometric} $d$-polytopes with $d+1+\alpha$ vertices or $d+1+\beta$ facets, respectively (see below for the precise definitions).

In particular, this shows that looking for upper bounds for the extension complexity of $d$-polytopes with $d+1+\alpha$ vertices
is not a good approach for finding upper bounds for the extension complexity of $d$-polytopes in terms of their number of vertices: 
When $d>(\frac{\alpha-1}{2})^2$ there are $d$-polytopes with $d+1+\alpha$ vertices (or facets) with extension complexity $d+1+\alpha$ (Corollary~\ref{cor:alphabound}).

Observe also that the bounds above, when specialized to $d=2$, give a lower bound of $2\sqrt{2n-2} - 1$ for the extension complexity of a generic (even rational) $n$-gon (Corollary~\ref{cor:polygonbound}).
This bound is tight for $n\leq15$ 
\cite{VandaeleGillisGlineurTuyttens2015}, 
but also for general $n$ up to a multiplicative constant, as the \emph{admissible $n$-gons} of Shitov show~\cite{Shitov2014a}.
This order of magnitude was already attained by the previous best lower bound of $\sqrt{2n}$ for the extension complexity of generic $n$-gons, by Fiorini, Rothvo{\ss} and Tiwary~\cite{FioriniRothvossTiwary2012}. However, their approach did not extend directly to the rational case, where they got a lower bound of order $\Omega(\sqrt{{n}/{\log n}})$.

The precise statement of our results uses the concept of 
\defn{realization space}
of a (combinatorial type of) polytope $P$, which is 
the set $\rs(P)$ of all polytopes that are combinatorially equivalent to $P$ (see \cite{RichterGebert1997} for an extensive monograph on realization spaces of polytopes). 
Parametrized by the vertex coordinates, or by the facet defining inequalities, it is a primary basic semi-algebraic set. We use $\rs_{\xc\leq K}(P)$ to deonote the 
subset containing those instances with extension complexity at most $K$.
\begin{restatable}{theorem}{thmdimrs}
 Let $P$ be a polytope whose %
 realization space has dimension~$r$, then 
 $\rs(P)\setminus \rs_{\xc\leq K}(P)$ is a full-dimensional dense semi-algebraic subset of $\rs(P)$ for every \[K <   2 \sqrt{r - d} -d + 1.\]
 In particular:
\begin{enumerate}
  \item For every $Q\in \rs(P)$ there is some polytope $Q'\in \rs(P)$ arbitrarily close to $P$ in Hausdorff distance such that 
  \begin{equation*}
  \xc(Q') \geq   2 \sqrt{r - d} -d + 1.
  \end{equation*}
  
  \item If $R$ is drawn randomly from a continuous probability distribution on~$\rs(P)$, then  almost surely
  \begin{equation*}
  \xc(R) \geq   2 \sqrt{r - d} -d + 1.
  \end{equation*}
\end{enumerate}
\end{restatable}

To recover the statements above, observe that the realization space of a simplicial $d$-polytope with $n$ vertices always has dimension~$dn$; and the realization space of a simple
$d$-polytope with $m$ facets always has dimension~$dm$. 
Moreover, polytopes with rational coordinates are dense in these realization spaces, and hence one can also impose the approximating polytope $Q'$ to be rational.

\section{Polytopes with at most $d+4$ vertices or facets}\label{sec:d+4vertices}

In this section we study the extension complexity of $d$-polytopes with at most $d+4$ vertices.
As a warm-up, we start with the classification of $d$-polytopes with $<d+4$ vertices, which is straightforward. Then we do the full characterization in Section~\ref{sec:d+4verticesmanyfacets}. We end with a  combinatorial characterization of $d$-polytopes with $d+4$ vertices and at most $d+3$ facets, which appear as a special family in our classification.

As for notation, $P$ will typically denote a $d$-dimensional polytope whose extension complexity $\xc(P)$ we want to understand, and $Q$ will be one of its extensions; a polytope for which there is a linear projection $\pi$ such that $\pi(Q)=P$. 

\subsection{Polytopes with less than $d+4$ vertices}\label{sec:lessthand+4vertices}

The cases of $d$-polytopes with $d+1$ or $d+2$ vertices are trivial.

\begin{lemma}\label{lem:xc=d+2}
The only $d$-polytope $P$ with $\xc(P)=d+1$ is the $d$-simplex $\Delta_{d}$, and the only with $\xc(P)=d+2$ are those with $d+2$ vertices or facets.
\end{lemma}
\begin{proof}
Any polytope of dimension higher than $d$ has at least $d+2$ facets, hence a $d$-polytope $P$ with $\xc(P)\leq d+2$ is either its own optimal extension and has $\xc(P)$ facets, or the projection of a $(d+1)$-simplex, in which case it has at most $d+2$ vertices.
\end{proof}

This settles the extension complexity of $d$-polytopes with $d+2$ vertices. 
Their complete characterization, which we will use later, is classical (see \cite[Section~6.1]{Gruenbaum} or \cite[Section 6.5]{Ziegler} for more details).

\begin{lemma}\label{lem:d+2char}
Every $d$-polytope with $d+2$ vertices (resp. facets) is projectively equivalent to $\ps$ (resp. $\pc$),
for some $k\geq 0$ and $n,m\geq 1$ with $k+n+m=d$.
\end{lemma}

Polytopes with $d+3$ vertices are also easy to analyze.

\begin{proposition}
 The extension complexity of every $d$-polytope $P$ with $d+3$ vertices is the 
 minimum of its number of vertices and facets.

 That is, $\xc(P)=d+3$ except if $P$ is combinatorially equivalent (and hence projectively equivalent)
 to $\pcar{d-3}{1}{2}$, the $(d-3)$-fold pyramid over the triangular prism, which has extension complexity~$d+2$.
\end{proposition}
\begin{proof}
By Lemma~\ref{lem:xc=d+2}, $P$ can only have $\xc(P)\leq d+2$ if it has $d+2$ facets.
By Lemma~\ref{lem:d+2char}, $P$ is projectively equivalent to
~$\pc$.
Now, $\pc$ has $k+n+m+2$ facets and $k+(n+1)(m+1)$ vertices. The only solutions in the natural numbers of the equation 
$$k+n+m+3 =k+(n+1)(m+1)$$
are $(k\geq 0, n=2, m=1)$ and $(k\geq 0, n=1, m=2)$.
Which means that $P$ is an iterated pyramid over a triangular prism.
\end{proof}
\subsection{Polytopes with $d+4$ vertices and at least $d+4$ facets}\label{sec:d+4verticesmanyfacets}

First of all, observe that we can focus on those polytopes that have at least $d+4$ facets,
because extension complexity is preserved by duality, and we have already studied
those with at most $d+3$ vertices. Up to taking pyramids, there are only finitely many combinatorial types of $d$-polytopes with $d+4$ vertices and at most $d+3$ facets~\cite{Padrol2016}, that we enumerate in Section~\ref{sec:d+4verticesfewfacets}.

The following easy observation, whose proof is analogue to that of Lemma~\ref{lem:xc=d+2}, shows that we have to focus on the case where the projection is of co-dimension one. 

\begin{lemma}\label{lem:onlyonelift}
 If a $d$-polytope $P$ with $d+4$ vertices and at least $d+4$ facets has an extension~$Q$
 with less than $d+4$ facets, then $Q$ is $(d+1)$-dimensional and has $d+3$ facets.
\end{lemma}

In order to study possible candidates for $Q$, we will use (a simplified version of) the Projection Lemma of Sanyal and Ziegler \cite{SanyalZiegler2010,Ziegler2004}, which can be also
understood in terms of McMullen's transforms and diagrams~\cite{McMullen1979}. 
It characterizes the faces preserved under projections.
If $P=\pi(Q)$ is the image of $Q$ under an affine projection $\pi$, we say that a proper face $F$ of $Q$ is 
\defn{preserved} by $\pi$ if $\pi(F)$ is a face of $P$ and $\inv{\pi}(\pi(F))=F$, and that it is 
\defn{strictly preserved} if moreover $\pi(F)$ and $F$ have the same dimension.

\begin{figure}[htpb]
\centering 
\includegraphics[width=.25\linewidth]{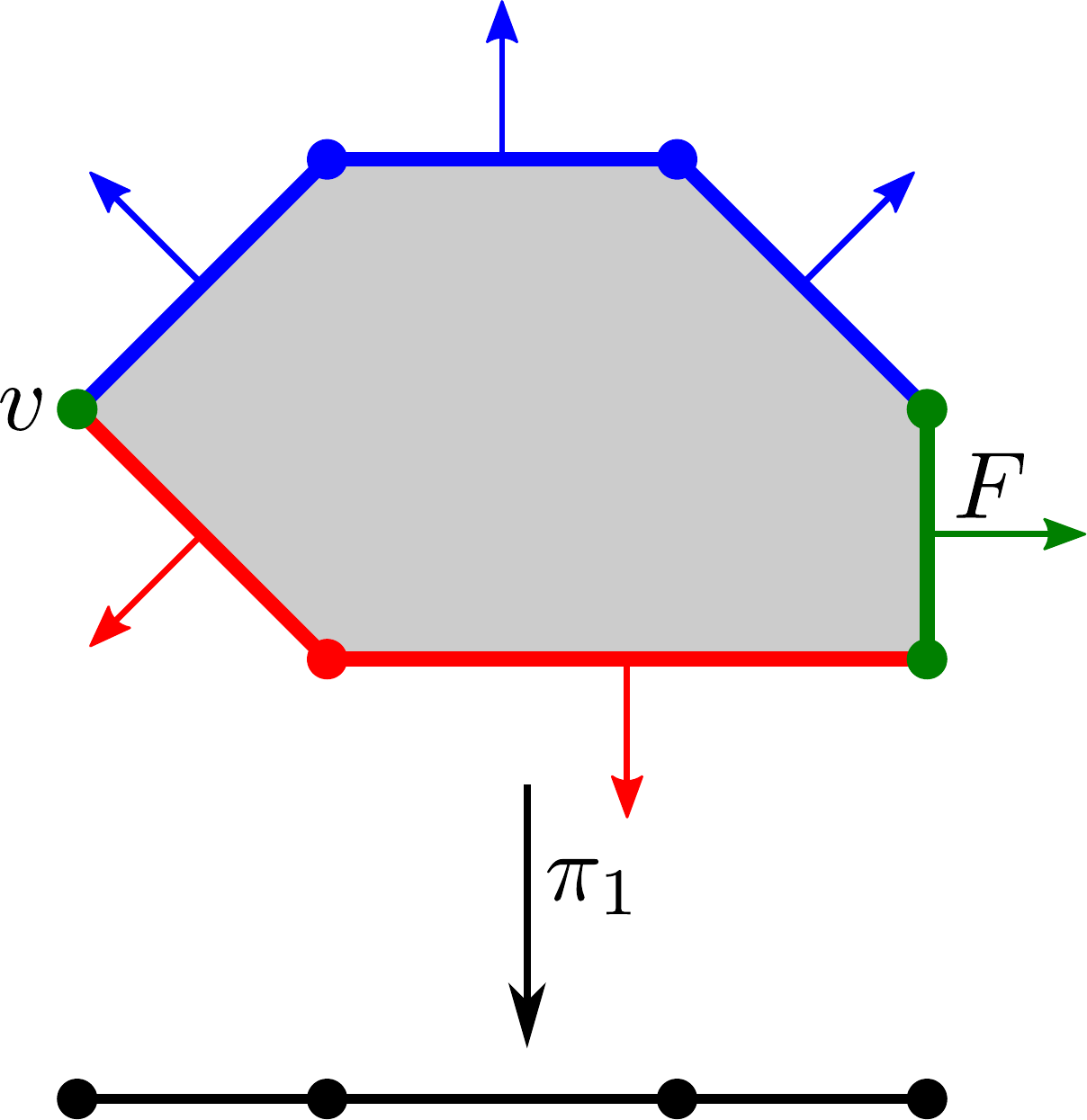}

\caption{A hexagon with 3 upper, 1 vertical and 2 lower facets. The vertex $v$ and the edge $F$ are preserved under $\pdp{1}$, and $v$ is strictly preserved. }\label{fig:projlemma}
\end{figure}

By Lemma~\ref{lem:onlyonelift}, we can always assume that the projection is the map $\pd:\RR^{d+1}\to\RR^d$ that forgets the last coordinate.
In this setup, the Projection Lemma can be expressed in the following terms.
A facet $F$ of a polytope $Q$ with
outer normal vector $n_F$ is \defn{upper}, \defn{lower} or \defn{vertical} according
to whether the last coordinate of $n_F$ is positive, negative or zero, respectively.
Two facets of $Q$ are \defn{complementary} if one of them is upper and the other lower.

\begin{lemma}[{Projection Lemma \cite[Lem.~2.5]{SanyalZiegler2010} \cite[Prop.~3.2]{Ziegler2004}}]\label{lem:projlem}
Let $F$ be a proper face of a $(d+1)$-polytope $Q$, and let $F_1,\dots,F_k$ be the facets of $Q$ containing $F$. Then $F$
is strictly preserved by $\pd$ if and only if 
$F_1,\dots,F_k$ contain a pair of complementary facets, and it is 
preserved but not strictly preserved if $F_1,\dots,F_k$
are all vertical.
\end{lemma}

In the example of Figure~\ref{fig:projlemma}, the vertex $v$ and the edge~$F$ are preserved. Since $v$ is the intersection of a pair of complementary facets, it is strictly preserved (it is the preimage of a face of the same dimension). In contrast, $F$ is only contained in a vertical facet (itself), so it is preserved but not strictly preserved (it is the preimage of a face of different dimension). None of the other faces of the hexagon are preserved.

The following three technical lemmas will ease our proof of Proposition~\ref{prop:d+4projection}.

\begin{lemma}\label{lem:2uplo}
Any minimal $(d+1)$-extension $Q$ of a $d$-polytope $P$ has at least two upper and two lower facets.
\end{lemma}
\begin{proof}
Since the normal vectors of $Q$ are positively spanning, $Q$ must have at least one upper and lower facet. If $F$ was the 
only upper/lower facet of $Q$, then $P$ would be affinely isomorphic to $F$, which has at least one facet less than~$Q$.
\end{proof}

We will work with projections of $(d+1)$-polytopes with $d+3$ facets, which by Lemma~\ref{lem:d+2char} are of the form $\pc$. In the next lemma we call a facet $F$ of a polytope $Q$ \defn{pyramidal} if it contains all the vertices but one; that is, if $Q$ is a pyramid with basis $F$.

\begin{lemma}\label{lem:allbut2}
If $Q=\pc$, $k+n+m=d+1$ is a minimal $(d+1)$-extension of a $d$-polytope $P=\pd(Q)$, then all but at most $2$ vertices of $Q$ are
strictly preserved under~$\pd$, 
and all preserved vertices are strictly preserved.

Moreover, if a vertex of $Q$ is not strictly preserved, then it must be the 
intersection of all the facets of $Q$ but $2$, which are the only $2$ upper or lower facets of $Q$, and are not pyramidal.
\end{lemma}
\begin{proof}
There are two kinds of vertices of $Q=\pc$: 
\begin{enumerate}
 \item $k$ apices of the pyramids, which are the intersection of all the facets but one, the corresponding pyramidal base;
 \item $(n+1)(m+1)$ vertices of the product, which are the intersection of all the facets but two, one arising from each simplex of the product.
\end{enumerate}

Recall that $Q$ has at least two upper and two lower facets by Lemma~\ref{lem:2uplo}. Hence, the Projection Lemma~\ref{lem:projlem} implies that every apex must be strictly preserved, because the defining 
facets must contain a complementary pair.  Moreover, again by the Projection Lemma~\ref{lem:projlem}, every preserved vertex must be also strictly preserved, since there are at least four non-vertical facets and every non-apex vertex belongs to all the facets but two.

If a vertex $v$ is not (strictly) preserved, then it must be a vertex of $\car{n}{m}$. It is the intersection of all 
facets but two, $F,G$, one coming from each simplex, and hence non-pyramidal. The only way that the facets incident to $v$ do not contain a complementary pair is that $F$ and $G$ are either the only two upper facets or the only two lower facets of $Q$. 
In particular, there are at most two such vertices, and there are two only when there are exactly two upper facets and two lower facets and all the remaining are vertical.
\end{proof}

We only need one final ingredient that uses the concept of \defn{equatorial ridge}, which is a ridge that is the intersection of two complementary facets of $Q$. For example, 
the vertex $v$ of Figure~\ref{fig:projlemma} is an equatorial ridge.
If $P=\pd(Q)$, the \defn{lift} of a face $F$ of $P$ is $\inv\pd(F)$, which is a face of $Q$. 

\begin{lemma}\label{lem:prolemrid}%
Let $Q$ be a $(d+1)$-extension of a $d$-polytope $P=\pd(Q)$. A face $F$ of $Q$ is a lift of a facet 
of $P$ if and only if it is either 
a vertical facet or
an equatorial ridge of $Q$.
\end{lemma}
\begin{proof}
Since the projection is of codimension $1$, the lift of a $(d-1)$-face of $P$ is either a not strictly preserved facet of $Q$ ($d$-dimensional), or a strictly preserved ridge ($(d-1)$-dimensional). The claim now follows from the Projection Lemma~\ref{lem:projlem}.
\end{proof}

We are ready to state and prove the main result of this section. %
\begin{proposition}\label{prop:d+4projection}
 Let $P$ be a $d$-polytope with $d+4$ vertices and at least $d+4$ facets with extension complexity $\xc(P)=d+3$.
 Then $P=\pd(Q)$, where $Q$ is a $(d+1)$-polytope combinatorially equivalent to $\pcar{d-2}{1}{2}$ and each vertex of $Q$ is strictly preserved under $\pi_d$.
 \end{proposition}
 \begin{proof}
 If $\xc(P)=d+3$, then $P=\pd(Q)$ for some $(d+1)$-polytope $Q\cong \pc$, where $k=d+1-n-m$, by Lemmas~\ref{lem:d+2char} and \ref{lem:onlyonelift}. Then
$Q$ has $nm+d+2$ vertices, of which all but $\ell\in \{0,1,2\}$ are strictly preserved, by Lemma~\ref{lem:allbut2}. Since all the vertices of $P$ are the image of strictly preserved
 vertices of $Q$ (again by Lemma~\ref{lem:allbut2}), we get that
 \[nm+d+2-\ell=d+4.\]
 The only solutions $(\ell,n,m)$ of this equation with $\ell\in \{0,1,2\}$ and $n,m\geq 1$ are: 
 $(0,1,2)$, $(1,1,3)$, $(2,1,4)$ and $(2,2,2)$.

We will show that all these cases are either impossible or can be reduced to the case $Q\cong \pcar{d-2}{1}{2}$, that is $(\ell,n,m)=(0,1,2)$.

If $Q\cong\pcar{d-3}{1}{3}$, then exactly one of its vertices, which we denote $v$, is not preserved. Observe that in this case $\pi(Q\setminus v)=\pi(Q)$ (where by $Q\setminus v$ we denote the polytope $\conv(\verts(Q)\setminus v))$). This vertex cannot be an apex since otherwise $P\cong Q\setminus v$, which has $d+2$ facets. Therefore, $v$ is a vertex of $\car{1}{3}$. 
But $\car{1}{3}\setminus v\cong \pcar{}{1}{2}$, and therefore this reduces to the case $Q\cong\pcar{d-2}{1}{2}$. We reduce analogously the case $Q\cong\pcar{d-4}{1}{4}$ to $Q\cong\pcar{d-3}{1}{3}$ and then to $Q\cong\pcar{d-2}{1}{2}$.

 It only remains the case $Q\cong  \pcar{d-3}{2}{2}$. In this case $\ell=2$ and we know by Lemma~\ref{lem:allbut2} that
 $Q$ has exactly two upper and two lower facets, which are non-pyramidal facets; %
 and the two non-preserved vertices are the intersection of the complements of 
 these two pairs of facets. Now we use Lemma~\ref{lem:prolemrid} to show that $P$ has at most $d+3$ facets, which contradicts our hypothesis. Indeed, the lift of a facet of $P$ can only be one of the $d-3$ pyramidal facets of $Q$, one of the $2$ non-pyramidal vertical facets, and one of the possible $4$ ridges arising as intersections of an upper and a lower facet of~$Q$. 
\end{proof}

We can do a more precise analysis of the combinatorial types of polytopes that arise this way.
For the first case, we need the following pair of folklore results concerning pyramids and hexagons.

\begin{lemma}\label{lem:pyr}
If $\pyr(P)$ is a pyramid with base $P$, then $\xc(\pyr(P))=\xc(P)+1.$
\end{lemma}
\begin{lemma}[{\cite[Prop.~4]{PadrolPfeifle2015}}]\label{lem:desarguian}
For a hexagon $P$, the following are equivalent:
\begin{enumerate}[label={(\roman*)}]
\item $\xc(P)=5$,
\item $P=\pi(Q)$ for $Q\cong\car{1}{2}$, and
\item $P$ is \defn{Desarguian}; that is, the lines $p_0\wedge p_1$, $p_5\wedge p_2$ and $p_3\wedge p_4$ are concurrent for some cyclic labeling of its vertices $p_0,\dots,p_5$.
\end{enumerate}
\end{lemma}

\begin{figure}[htpb]
\centering 
\includegraphics[width=.45\linewidth]{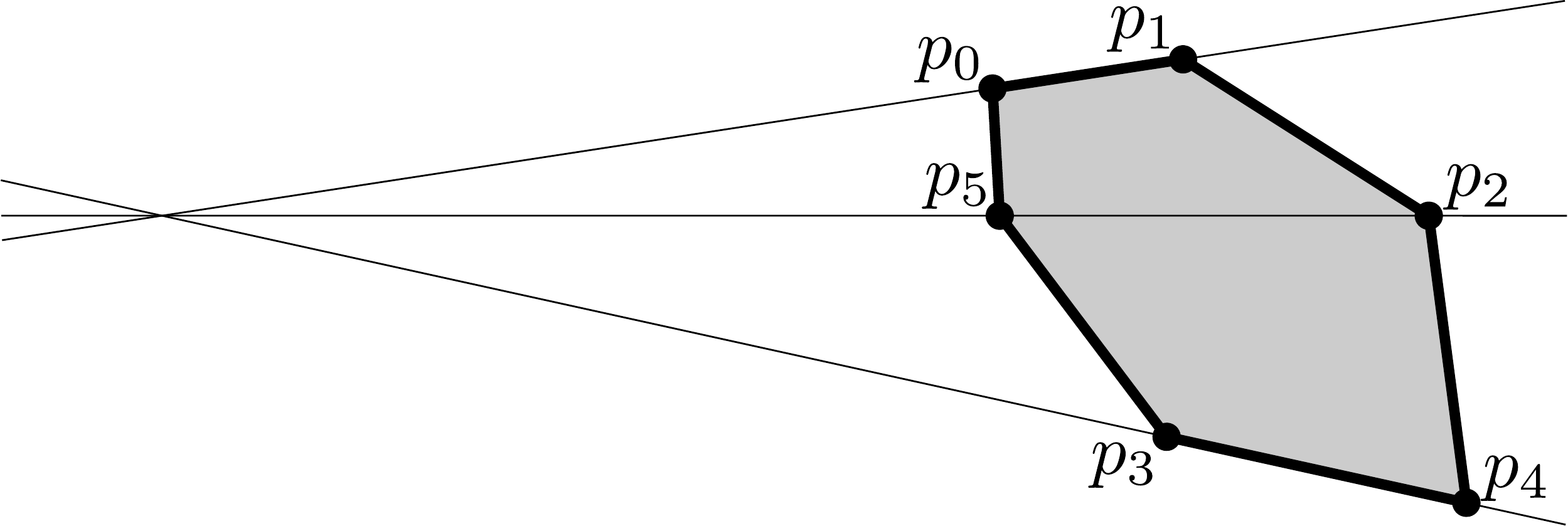}

\caption{A Desarguian hexagon.}
\end{figure}

\begin{proposition}\label{prop:d+4classification}
 A $d$-polytope $P$ with $d+4$ vertices and at least $d+4$ facets has
 extension complexity $\xc(P)=d+3$ if and only if either it
  \begin{enumerate}[label=(\arabic*)]
 \item\label{it:deshex} is a $(d-2)$-fold pyramid over a Desarguian hexagon~$H$, or
 \item\label{it:tripri}  has a subset of $6$ vertices forming a triangular prism. 
\end{enumerate}
\end{proposition}

\begin{proof}
According to Proposition~\ref{prop:d+4projection}, $P=\pd(Q)$ for some $(d+1)$-polytope $Q\cong \pcar{d-2}{1}{2}$ that has
$d+4$ vertices: six of which, $q_1,\dots, q_6$, span a triangular prism $\conv\{q_1,\dots,q_6\}\cong\car{1}{2}$ and the
remaining $d-2$ are apices of pyramids.
The polytope $P=\pd(Q)$ has also $d+4$ vertices. We denote  $b_i=\pd(q_i)$ and $B=\conv\{b_1,\dots,b_6\}$. 

If $\dim(B)=2$, then each of the remaining $d-2$ vertices must contribute to increase the dimension by one, and hence $P$ is a $(d-2)$-fold pyramid over a the hexagon $B$.  By Lemma~\ref{lem:desarguian}, a hexagon is a projection of $\car{1}{2}$ if and only if it is Desarguian. So, with the aid of Lemma~\ref{lem:pyr}, we recover~\ref{it:deshex}.

If $\dim(B)=3$, then $B\cong \car{1}{2}$, because $\pd$ restricts to an affine isomorphism on the triangular prism. We get therefore~\ref{it:tripri}.

The reciprocal statements, that polytopes fulfilling \ref{it:deshex} or \ref{it:tripri} have extension complexity at most $d+3$, are direct because adding a vertex can only increase extension complexity by~$1$.
\end{proof}

The second point can be developed even further. For that, we need some extra notation. The \defn{join} of two polytopes $P$ and $Q$, denoted $P \join Q$, which consists in taking the convex hull 
of a copy of $P$ and $Q$ placed in skew affine subspaces (see \cite[Sec.~15.1.3]{HenkRichterGebertZiegler1997}). A
\defn{Lawrence extension} of a $d$-polytope $P$ at a point $p\in\RR^d$ is 
$$\conv(P\cup (p+e_{d+1})\cup (p+2e_{d+1}))\subset \RR^{d+1},$$ and a
\defn{one-point-suspension} of $P$ at $p$ is $$\conv(P\cup (p+e_{d+1})\cup (p-e_{d+1}))\subset \RR^{d+1}$$ (cf. \cite[Sec.~4.2.5 \& Sec.~5.5.3]{DeLoeraRambauSantos2010}).

\begin{lemma}\label{lem:d+4classification}
Any $d$-polytope $P$ with $d+4$ vertices six of which span a triangular prism is an iterated pyramid over either
\begin{enumerate}[label=(\arabic*)]
 \item\label{it:joinprod} the join of $\car{1}{2}$ with a direct sum of simplices, or
  \item\label{it:suspLaw} a polytope obtained by a sequence of {one-point-suspensions} and {Lawrence extensions} on a (projective) point added to $\pcar{k}{1}{2}$. In this case, $P=\conv(Q\cup p)$, where $Q\cong \pcar{k}{1}{2}$ is the convex hull of the vertices of $P$ distinct from~$p$. 
\end{enumerate}
\end{lemma}
\begin{proof}
With the notation of the previous theorem, we denote $b_1,\dots, b_6$ the vertices of $P$ spanning the triangular prism and by $a_1,\dots, a_{d-2}$ the remaining ones. Set $A=\conv\{a_1,\dots,a_{d-2}\}$ and $B=\conv\{b_1,\dots,b_6\}$. Since $d=\dim(P)\leq \dim(A)+\dim(B)+1$, $A$ is either $d-3$ or $d-4$ dimensional.

If $\dim(A)=d-4$, then $A$ must be combinatorially equivalent to an iterated pyramid over a product of simplices by Lemma~\ref{lem:d+2char}, because it has $d-2$ vertices. Moreover, since $\dim(B)=3$, $A$ and $B$ must lie in skew affine subspaces, and hence $P=A\join B$.  Therefore, there are $k',n',m'$ with $k'+n'+m '=d-4$ such that $P\cong \pyr_{k'}((\car{1}{2})\join(\dsum{n'}{m'}))$.

Otherwise, $\dim(A)=d-3$. We can assume without loss of generality that $P$ is not a pyramid, which is equivalent to assume that no $a_i$ lies in all the facets of $P$ but one.
For the sake of the induction, we remove the condition that the all the points $a_i$ are vertices of $P$. Hence, we consider $d$-dimensional configurations of points $\{a_1,\dots,a_{d-2}, b_1,\dots,b_6\}$ such that
\begin{itemize}
\item $\conv(b_1,\dots,b_6)\cong\car{1}{2}$,
\item $\conv\{a_1,\dots,a_{d-2}\}\cong \Delta_{d-3}$, and
\item there is no hyperplane containing all the points but one.
\end{itemize}

Then we claim that for $d\geq 4$, any such a configuration can be obtained from one in one dimension less by the operation of one-point-suspension or Lawrence extension.

Consider the affine hull of $\{b_1,\dots,b_6,a_1,\dots,a_{d-4}\}$. It is a hyperplane $H$ because otherwise the points $\{b_1,\dots,b_6,a_1,\dots,a_{d-4}, a_{d-3}\}$ would be contained in a common hyperplane  and $P$ would be a pyramid. For the same reason, neither $a_{d-2}$ or $a_{d-3}$ belong to $H$.
Let $a_{d-3}'$ be the (projective) point of intersection of the line ${a_{d-2}\wedge a_{d-3}}$ with $H$. If $a_{d-2}$ and $a_{d-3}$ lie at the same side of $H$, then $\{b_1,\dots,b_6,a_1,\dots,a_{d-2}\}$
is the Lawrence extension of $\{b_1,\dots,b_6,a_1,\dots,a_{d-4}\}$ at $a_{d-3}'$. Otherwise, 
if $H$ separates $a_{d-2}$ and $a_{d-3}$, then it is its one-point-suspension at $a_{d-3}'$.
\end{proof}

To finish this subsection, we show that among all possible combinatorial types of $d$-polytopes with $d+4$ vertices, the vast majority does not have any realization with extension complexity smaller than $d+4$.
\begin{corollary}\label{cor:number}
The number of (unlabeled) combinatorial types of $d$-polytopes with $d+4$ vertices that have a realization $P$ with $\xc(P)=d+3$ is $\theta(d^ 2)$.
\end{corollary}
\begin{proof}
As we will see in Proposition~\ref{prop:d+4verticesd+3facets}, for each $d$ there are at most $9$ combinatorial types of $d$-polytopes with $d+4$ vertices and at most $d+3$ facets. The remaining polytopes with $d+4$ vertices and extension complexity $d+3$ belong hence to one of the two families treated in Lemma~\ref{lem:d+4classification}.

There are exactly $\ffloor{(d-4)^2}{4}$ distinct combinatorial types of polytopes of the form $\pyr_{k'}((\car{1}{2})\join(\dsum{n'}{m'}))$ with $k'+n'+m'+2=d-4$ (follows from \cite[Theorem~6.1.4]{Gruenbaum}).

It only remains to count those arising as Lawrence extensions and one-point-suspensions.
Instead of counting combinatorial types of polytopes, we count possible oriented matroids with this combinatorial type of convex hull, which is a finer combinatorial invariant~\cite{RichterGebertZiegler1997}.
The oriented matroid of a Lawrence extension or a one-point-suspension depends only on the oriented matroid of the original point configuration and the point where it is applied.

Let $C$ be the number of combinatorially different ways to add a projective point to $\car{1}{2}$ (i.e. the number of different single element extensions as an oriented matroid). Since Lawrence extensions, one-point-suspensions and taking pyramids commute, we only need to know the placement of the original extra point and how many times we performed each operation. This is the number of triples $(\ell,o,p)$ of nonnegative integers adding to $d-3$, \ie $\binom{d-1}{2}$. So in total we have $C\frac{(d-1)(d-2)}{2}$ possible oriented matroids for $P$.
Since not all these oriented matroids will represent points in convex position, and not all will give combinatorially distinct polytopes, this is just an upper bound.
\end{proof}

\subsection{Polytopes with $d+4$ vertices and at most $d+3$ facets}\label{sec:d+4verticesfewfacets}

As we have already mentioned, any $d$-polytope with at most $d+3$ facets has extension complexity not larger than $d+3$. We provide a full classification of those $d$-polytopes that have $d+4$ vertices and at most $d+3$ facets.

It is proven in \cite{Padrol2016} that there is a function $\Dab$ such that every $d$-polytope with at most $d+1+\alpha$ vertices and at most $d+1+\beta$ facets is a $k$-fold pyramid over another such a polytope in dimension at most~$\Dab$; and that $\Dab$ satisfies
$\Dab\leq \min\left\{\F{\aa}{\bb},\F{\bb}{\aa}\right\}$,
where
\[\F{x}{y}=
\begin{cases}
0 & \text{ if }x=0\text{ or }y=0,\\
3x+y-2 & \text{ if }1\leq x\leq 5,\\
\binom{x}{2}+y+3& \text{ if }x\geq 5.
\end{cases}\]

In particular, any $d$-polytope with $d+4$ vertices and at most $d+3$ facets is a $k$-fold pyramid over a $d$-polytope with $d+4$ vertices in dimension at most~$7$. Such polytopes with $d+2$ facets are easy to classify using Lemma~\ref{lem:d+2char}, and those with $d+3$ facets are also easy to deal with using Gale diagrams, because its polar is a $d$-polytope with $d+3$ vertices, and hence has a planar Gale diagram, see~\cite[Section~6.3]{Gruenbaum}.
A small computer enumeration certifies that there are no $7$-polytopes with $11$ vertices and $10$ facets that are not pyramids, and gives us the complete classification of those in dimension less or equal to~$6$.

\begin{figure}[htpb]
\centering 
 \includegraphics[width=.9\linewidth]{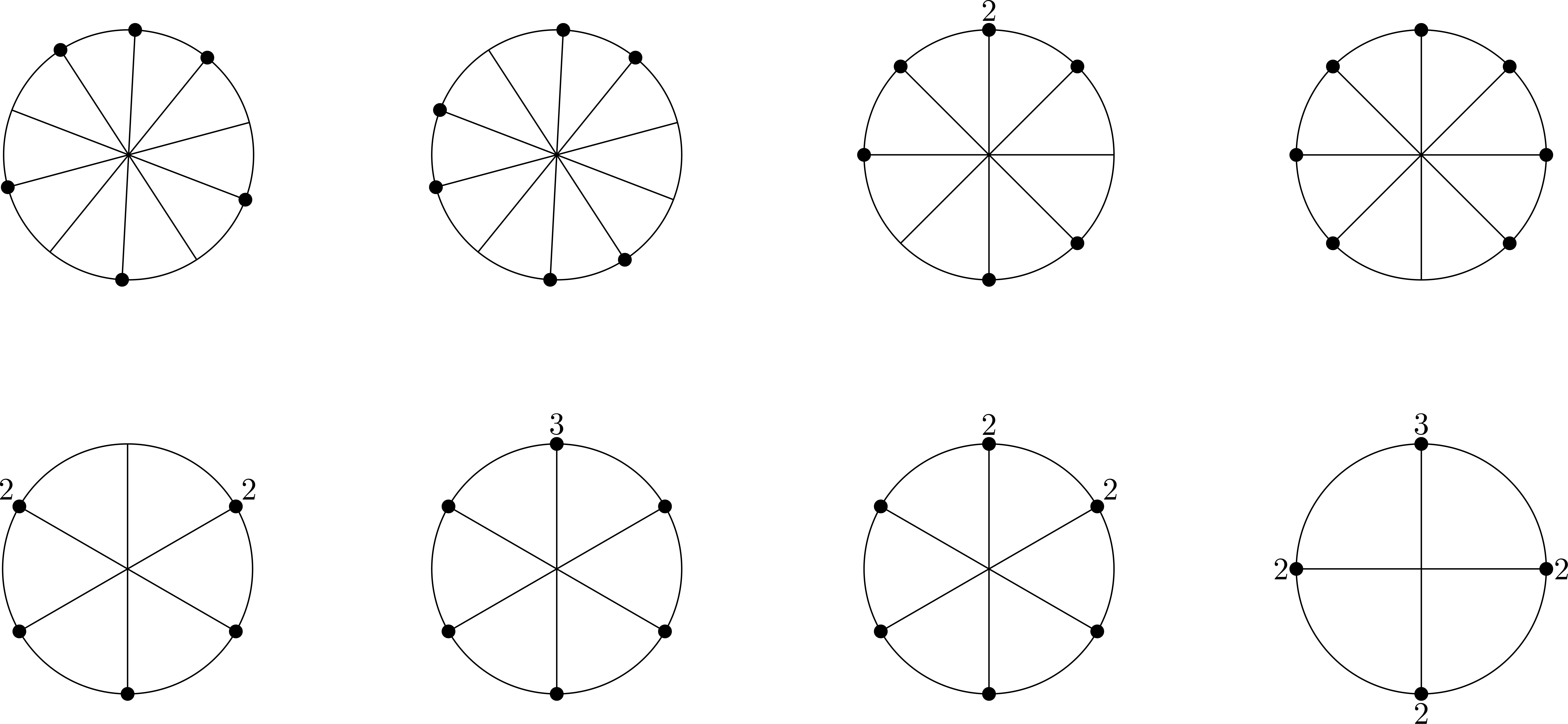}
 
 \caption{Gale diagrams of all $d$-polytopes with $d+3$ vertices and $d+4$ facets that are not a pyramid. (Numbers represent multiplicities, see~\cite[Section~6.3]{Gruenbaum}.)}\label{fig:GDd+4verticesd+3facets}	
\end{figure}

\begin{proposition}\label{prop:d+4verticesd+3facets}

Let $P$ be a $d$-polytope with $d+4$ vertices that is not a pyramid. Then,
\begin{itemize}
 \item $P$ has $d+2$ facets if and only if it is combinatorially equivalent (and hence projectively equivalent) to $\car{1}{3}$;
 \item $P$ has $d+3$ facets if and only if it is combinatorially equivalent to one of the eight polytopes whose polar polytope has a Gale diagram depicted in Figure~\ref{fig:GDd+4verticesd+3facets}, which are of respective dimensions $3$, $3$, $4$, $4$, $4$, $5$, $5$ and $6$.
\end{itemize}
\end{proposition}

\section{Lower bounds for extension complexity}\label{sec:bounds}

In this section we study realizations spaces of polytopes~\cite{RichterGebert1997}. Then a simple count of degrees of freedom 
will provide us with instances of polytopes with ``large'' extension complexity.

Any $d$-polytope with $n$ facets containing the origin in its (proper) interior is of the form
\[
 P_A=\set{x\in \RR^d}{Ax\leq \vecone},
\]
for some $n\times d$ matrix $A$ ($\vecone$ is the all-ones column vector). This representation is unique up to row permutations, and hence unique if we consider the facets as being labeled.

Let $P$ be a (combinatorial type of) $d$-polytope with $n$ facets. We consider the set $\rs(P)$ of all (facet-labeled) polytopes combinatorially isomorphic to $P$ that contain the origin in the interior.
Parametrized by the facet inequalities, this is the semi-algebraic set:
\[
 \rs(P)=\set{A\in \RR^{n\times d}}{ P_A \cong P}\subset \RR^{n\times d}.
\]
While this is not the standard definition of realization space (one usually considers its quotient by the set of affine or projective transformations, and not only those realizations containing the origin),  
this version suffices for our purposes and simplifies the exposition. Note that extension complexity is invariant under projective transformations, in particular translations, and hence assuming that the polytope
contains the origin is not a restriction.

We also consider two subsets of $\rs(P)$:
\begin{align*}
 \rs_{N,D}(P)&=\set{A\in \rs(P)}{P_A =\pi(Q) \text{ for some $D$-polytope $Q$ with $N$ facets}},\\
 \rs_{\xc\leq K}(P) &=\set{A\in \rs(P)}{\xc(P_A)\leq K}.
 \end{align*}
The first, $\rs_{N,D}(P)$, contains all realizations of $P$ that arise as projections of $D$-polytopes with $N$ facets. The second, $\rs_{\xc\leq K}(P)$, contains those realizations whose extension complexity is bounded by $K$. Thus, it can also be defined as the union
\begin{equation}\label{eq:union}\rs_{\xc\leq K}(P)=\bigcup_{\substack{N\leq K\\d\leq D\leq N-1}}\rs_{N,D}(P).\end{equation}

Our main result in this section is an estimation of the dimension of $\rs_{\xc\leq K}(P)$: as long as its dimension is smaller than that of $\rs(P)$, there will be instances of $P$ with extension complexity larger than $K$. Note that here we talk about two different dimensions that should not be confused: the (affine) dimension~$d$ of $P$, and the (algebraic) dimension of its realization space, which is that of the smallest algebraic variety of $\RR^{n\times d}$ that contains it. We refer to~\cite[Chapter~2]{BochnakCosteRoy1998} for the concepts concerning semi-algebraic sets appearing in the proof.

\thmdimrs*

\begin{proof}%

 \newcommand{\pDd}{\pi}

 Let $\pDd:\RR^D\to\RR^d$ be the projection that forgets the last $D-d$ coordinates, and consider the set of $D$-dimensional polytopes with $N$ facets that contain the origin in their interior and are mapped onto a polytope combinatorially equivalent to $P$
\[\cQ_{N,D}(P):=\set{B\in \RR^{N\times D}}{\pDd(P_B)\cong P}\subset\RR^{N\times D}.\]
 Its dimension is obviously at most $DN$.
 
 By construction, there is a natural surjective map $\varphi:\cQ_{N,D}(P)\to \rs_{N,D}(P)$ that maps each matrix $B\in\cQ_{N,D}(P)$ to the matrix $A$ such that $P_A=\pDd(P_B)$. 
 The only subtlety in making this map explicit lies in deciding the order of the inequalities defining $P_A$; but this can be done separately for each combinatorial type of $P_B$, in such a way that 
 the map is semi-algebraic.

 If $A$ belongs to $\rs_{N,D}(P)$, then its preimage $\varphi^{-1}(A)$ is a semi-algebraic 
 subset of $\cQ_{N,D}(P)$ (see~\cite[Prop.~2.2.7]{BochnakCosteRoy1998}). Its dimension is at least $(D+1)(D-d)$. 
 Indeed, if $Q$ is a $D$-polytope such that $\pDd(Q)=P_A$, and $\tau:\RR^D\to\RR^D$ is an affine transformation that preserves the first $d$ coordinates, then applying $\tau$ does not change its image under $\pDd$, \ie $\pDd(\tau(Q))=P_A$.
 The space of such affine transformations has dimension $(D+1)(D-d)$, since they are defined as $x\mapsto Mx+t$ for a matrix $M\in\RR^{D\times D}$ of the form
 \[
       M= \left(
        \begin{array}{c c c c c c}
                1 & \cdots & 0 & 0 & \cdots & 0 \\
                \vdots & \ddots & \vdots & \vdots & \ddots& \vdots\\
                0  &\cdots & 1 & 0 & \cdots & 0 \\
		\ast & \cdots & \ast & \ast & \cdots & \ast \\
		\vdots & \ddots & \vdots & \vdots & \ddots& \vdots\\
		\undermat{D}{\ast & \cdots & \ast & \ast & \cdots & \ast}
                \end{array}
    \right)
    \begin{array}{l@{}}
    \left.\begin{array}{@{}c@{}}\null\\\null\\\null\end{array}\right\}~d\\
    \left.\begin{array}{@{}c@{}}\null\\\null\\\null\end{array}\right\}~D-d
  \end{array}\vspace{.5cm}
    \] 
 and a vector $t\in\RR^D$ with zeros in its first $d$ coordinates. Since these affine transformations act freely on $\varphi^{-1}(A)$, the dimension of (each irreducible component of) $\varphi^{-1}(A)$ is at least $(D+1)(D-d)$.
 
 Hence, we can bound the dimension of $\rs_{N,D}(P)$ as follows (see~\cite[Thm.~9.3.2]{BochnakCosteRoy1998} and \cite[Cor.~4.2]{Coste2000}):
 \begin{equation}\label{eq:bounddim}\dim(\rs_{N,D}(P))\leq DN-(D+1)(D-d).\end{equation}
                                                                          
 Therefore, a necessary condition for $\dim(\rs_{N,D}(P))=\dim(\rs(P))$ is that 
 \begin{equation}\label{eq:boundN}
 N\geq \frac{(D+1)(D-d) + r}{D}.
 \end{equation}

 By~\eqref{eq:union}, $\rs_{\xc\leq K}(P)$ is a full-dimensional subset of $\rs(P)$ if and only if some $\rs_{N,D}(P)$ with $N\leq K$ is full-dimensional. Now, the minimum of the function $f(D)=\frac{(D+1)(D-d) + r}{D}$ is attained at $D=\sqrt{r - d}$. Therefore, by~\eqref{eq:boundN}, unless 
 $$K\geq f(\sqrt{r - d})=2 \sqrt{r - d} -d + 1$$
 $\rs_{\xc\leq K}(P)$ cannot be full-dimensional.
\end{proof}

The realization space of any simple $d$-polytope with $n$ facets has dimension~$dn$, because small perturbations of the facets keep the combinatorial type. Similarly, the realization space of any simplicial $d$-polytope with $n$ vertices has dimension $dn$, because its facets can be derived from the vertex coordinates, and the $n$ vertices can be perturbed without changing the combinatorial type.
\begin{corollary}\label{cor:lbxc}
 If $P$ is a simple $d$-polytope with $n$ facets, or a simplicial $d$-polytope with $n$ vertices, then every generic realization of $P$ has extension complexity at least $2 \sqrt{d n - d} -d + 1$.
\end{corollary}

Now we switch back to polytopes with few vertices/facets. The previous corollary shows us the existence of $d$-polytopes with $d+\alpha+1$ vertices and extension complexity at least $2 \sqrt{d(d+\alpha)} -d + 1$. Obviously, this lower bound is never larger than the trivial upper bound for extension complexity: 
\[2 \sqrt{d^2+d\alpha} -d + 1\leq 2 \sqrt{d^2+d\alpha+\frac{\alpha}{4}} -d + 1=2 \left(d+\frac{\alpha}{2}\right) -d + 1=d+\alpha+1.\]
However, it can get arbitrarily close if we fix $\alpha$ and let $d$ grow. In particular, if $d>(\frac{\alpha-1}{2})^2$, then 
\begin{align*}
2 \sqrt{d^2+d\alpha} -d + 1&=2 \sqrt{d^2+d(\alpha-1)+d} -d + 1\\
&>2 \sqrt{d^2+d(\alpha-1)+\left(\frac{\alpha-1}{2}\right)^2} -d + 1\\&=2 \left(d+\frac{\alpha-1}{2}\right) -d + 1=d+\alpha. 
\end{align*}

\begin{corollary}\label{cor:alphabound}
For any $d>(\frac{\alpha-1}{2})^2$ there are $d$-polytopes with $d+1+\alpha$ vertices (and $d$-polytopes with $d+1+\alpha$ facets) with extension complexity $d+1+\alpha$.
\end{corollary}

If we specialize the bound of Corollary~\ref{cor:lbxc} to polygons, we get the result we alluded to in the introduction. It improves slightly the previous best lower bounds for the worst extension complexity of an $n$-gon~\cite{FioriniRothvossTiwary2012}.

\begin{corollary}\label{cor:polygonbound}
The extension complexity of a generic (even rational) $n$-gon is at least $2\sqrt{2n-2} - 1$.
\end{corollary}

\section*{Acknowledgements}
Many ideas for Section~\ref{sec:bounds} grew out from interesting conversations with Julian Pfeifle while working on~\cite{PadrolPfeifle2015}. I also want to thank Raman Sanyal and G\"unter M.\ Ziegler for many stimulating discussions and valuable insights.

\bibliographystyle{siamplain}
\bibliography{few}

\end{document}